\documentclass[10pt,twoside]{article}
\usepackage{amssymb}

\usepackage{amsmath}
\usepackage{latexsym}
\usepackage{eget}
\newcommand{\erw}{elephant random walk}
\newcommand{\cfn}{{\cal F}_n}
\newcommand{\cfm}{{\cal F}_m}

\newcommand{\cgn}{{\cal G}_n}

\newcommand{\ind}{ 1\hspace{-1mm}1}
\newcommand{\mem}{{\mathfrak M}}
\newcommand{\memn}{{\mathfrak M}_n}
\newcommand{\xnp}{X_{n+1}}

\newcommand{\xkm}{X_{k-1}}

\begin{document}
\date{}
\title{\textsf{The Elephant random walk with gradually increasing memory}}
\author{Allan Gut\\Uppsala University \and Ulrich Stadtm\"uller\\
Ulm University}
\maketitle

\begin{abstract}\noindent
In the simple random walk the steps are independent, viz., the walker has no memory. In contrast, in the \erw\ (ERW), which was introduced by Sch\"utz and Trimper \cite{erwdef} in 2004, the next step always depends on the whole path so far. Various authors have studied further properties of the ERW. In  \cite{111} we studied the case when the Elephant remembers only a finite part of the first or last steps. In both cases there was no separation into two different regimes as in the classical ERW. We also posed the question about what happens if she remembers a gradually increasing past. This paper will give some answers to that question. We also discuss related questions for ERW:s with delays.
\end{abstract}

\fot{60F05, 60G50,}{60F15, 60J10}{Elephant random walk, delay, number of zeroes, law of large numbers, central limit theorem, Markov chain}
{ERW with increasing memory}

\section{Introduction}
\setcounter{equation}{0}
\markboth{A.\ Gut and U.\ Stadtm\"uller}{Elephant random walk}

In the classical \emph{simple\/} random walk the steps are equal to plus or minus one and independent---$P(X=1)=1-P(X=-1)=p$, ($0<p<1$); the walker has no memory.  Motivated by applications in physics, although interesting in its own right, is the so called \erw\ (ERW),  for which every step depends on the whole process so far. The ERW  was introduced in \cite{erwdef} in 2004, the name being inspired by the fact that elephants have a very long memory. Limit laws were formally proved in \cite{bercu,Coletti} using martingale theory. Actually, in the earlier papers \cite{Heyde, James} such results  were shown for so-called correlated Bernoulli processes. In \cite{111} we studied the case when the elephant has a restricted memory; assuming that future steps depend only on some distant past, only a recent past, or a mixture of both. Inspired by a suggestion in \cite{bercu2} we allowed, in \cite{113}, the possibility of delays in that the elephant, in addition, always has a choice of staying put. 

Formally, the \erw\ is defined as a  random walk in which the first step $X_1$ equals 1 with probability $s\in [0,1]$ and  to $-1$ with probability $1-s$, where, for convenience, we assume that $s=p$. After $n$ steps, at position $S_n=\sumk X_k$, the next step is defined as
\bea\label{20}\xnp=\begin{cases} +X_{K},\ttt{with probability} p\in[0,1], \\-X_{K},\ttt{with probability} 1-p,\end{cases}\eea
where $K$ has a uniform distribution on the integers $1,2,\ldots,n$. 

In that notation we studied, in \cite{111}, the case where $K$ in (\ref{20}) is uniformly distributed over the set of points constituted by a restricted memory. In particular we considered the case where for some fixed $m$ the random variable $K$ is uniformly distributed over the memories $\mem=\{1,2,\dots,m\}$  or $\mem=\{n-m+1,n-m+2,\dots ,n\}$ 
or a mixture of both. In these cases there was no separation in a diffusion part and a superdiffusion part depending on the value of $p$ as in the ordinary ERW, for which the transitional $p$-value equals $3/4$. In \cite{111} we posed the question what would happen if $m=m_n\nearrow \infty$. Some numerical simulations concerning this situation can be found in \cite{csv, moura,scv}. The present paper will give some theoretical answers to that question. 
  
The organization of the paper is as follows: After presenting the setup in the following section, we treat the main case, for which $m_n\nearrow \infty$, such that $m_n/n \to 0$ in Section 3. The next Section extends the results to delayed walks. In Section 5 we discuss the number of zeros occuring in this case, and in Section 6 we present moments when $m_n/n \to \alpha \in (0,1]$. The paper closes with some comments.

\section{The setup}\label{setup}
\setcounter{equation}{0}

We begin by assuming that the elephant remembers the first $m =m_n$ steps, i.e., that ${\cal F}_n =\sigma\{X_1,X_2, \dots,X_m\}$, where now the quantity $m$ depends on time and  
is increasing to infinity but not too fast, i.e., we assume that we have a triangular scheme of random variables $X_m=X_{n,m}$ and $\sigma$-algebras
\begin{equation}\label{mn}
{\cal F}_{m_n}=\sigma\{X_{n,1},X_{n,2},\dots,X_{n,m_n}\},\ttt{where}m_n \to \infty \mbox{ and } \frac{m_n}{n} \to 0 \ttt{as}  \nifi\,.
\end{equation}
For convenience we shall suppress the subindex $n$, viz, writing $m=m_n$ and  $S_m=S_{m_n}$. 

We begin by computing moments. Note that the beginning of the walk   is following an ERW with full memory, for which we can borrow from, e.g.,  \cite{bercu} or \cite{Coletti}. Based on these results we start with the calculation of moments. In contrast to the case of fixed $m$ we are now facing three different regimes, namely $0<p<3/4$, $p=3/4$, and $3/4<p\le1$, respectively, as in the case with full memory. We also remark that the ERW reduces to a simple, symmetric random walk (a coin-tossing random walk) when $p=1/2$.

Before we start we need to keep the following in mind: When the elephant remembers the whole past the 
$\sigma$-algebra ${\cal G}_n=\sigma\{X_1,X_2,\dots, X_n\}$  is relevant for $X_{n+1}$ and (see, e.g.,  \cite{bercu}) the behavior of the next step is governed by the relation $E(\xnp\mid \cgn) =(2p-1)\cdot\frac{S_n}{n}$, and hence $E(S_{n+1}\,|\, \cgn)=\gamma_n S_n$ where $\gamma_n=(n+2p-1)/n$, for $n\in \mathbb{N}$. This implies that $M_n=a_nS_n$ with $a_n=\Pi_{k=1}^n\gamma_k^{-1}$ is a martingale which allows one to prove various limit theorems which, however, depend on the value of $p$ with the critical value $p=3/4$.

Our next tool is an analog for ERW:s with a restricted memory. Toward that end, let  $\{\cfn,\,n\geq1\}$ denote the   $\sigma$-algebras generated by the memory  $\memn$, i.e., ${\cal F}_n=\sigma\{X_i : i \in \memn \}$.  Then,
\bea\label{uli}
E(\xnp\mid \cfn)
=(p-q)\cdot\frac{\sum_{i\in \memn}X_i}{|\memn|} = (2p-1)\cdot\frac{\sum_{i\in \memn}X_i}{|\memn|},
\eea
noticing that only the first equality holds in the delayed case.

When we condition on steps that are not contained in the memory it means that the elephant does not remember them, and, hence, cannot choose among them in a following step. Thus, if $A\subset \{1,2,\ldots,n\}$ is an arbitrary set of indices,  then
\bea\label{uli2}
E(\xnp\mid \sigma\{X_i\,, \, i \in A\cup \memn\})=E(\xnp\mid\cfn)=(p-q)\frac{\sum_{i\in \memn}X_i}{|\memn|}= (2p-1)\frac{\sum_{i\in \memn}X_i}{|\memn|},
\eea
where, again, the first equality only holds in the delayed case.

Finally, in order to avoid special effects we assume throughout that $0<p<1$; note that $p=1$ corresponds to $X_n=X_1$ for all $n$, and $p=0$ to the case of alternating summands. Also $\delta_a$ is the dirac-measure at the point $a$, and  $c$ and $C$ denote numerical constants which may change from line to line.

\section{ERW:s with an increasing memory;\quad {$\memn = \{1,2,\ldots,m_n\}$}}\label{drei}
\setcounter{equation}{0}
We begin with some facts on the asymptotics of mean and variance.

\begin{proposition}\label{moments} Let $m,n\to\infty$, such that $m/n\to0$ as $\nifi$.\\[1.5mm]
\textbf{\emph{(a)}} If $0<p<3/4$, then
\beaa
E\Big(\frac{S_n\sqrt{m}}{n}\Big)&=&(2p-1)E(\frac{S_m}{\sqrt{m}})+2(1-p)\frac{m}{n} E(\frac{S_m} {\sqrt{m}})\to 0,\\
\var\Big(\frac{S_n\sqrt{m}}{n}\Big)&=&(2p-1)^2 \var(\frac{S_m}{\sqrt{m}})
+{\cal O}\Big(\frac{m}{n}\Big)=\frac{(2p-1)^2}{3-4p} +o(1).
\eeaa
\textbf{\emph{(b)}} If $p=3/4$, then
\beaa
E\Big(\frac{S_n\sqrt{m/\log m}}{n}\Big)&=&\frac12E(\frac{S_m}{\sqrt{m \log m}})
+\frac12\frac{m}{n} E(\frac{S_m} {\sqrt{m\log m}})\to 0,\\
\var\Big(\frac{S_n\sqrt{m/ \log m}}{n}\Big)&=&\frac14 \var(\frac{S_m}{\sqrt{m \log m}})\left(1+{\cal O}\Big(\frac{m}{n}\Big)\right)=\frac14 +o(1).
\eeaa
\textbf{\emph{(c)}} If $3/4<p<1$, then
\beaa
E\Big(\frac{S_n\,m^{2(1-p)}}{n}\Big)&=&(2p-1)E(\frac{S_m}{m^{2p-1}})+2(1-p)\frac{m}{n} E(\frac{S_m }{m^{2p-1}})\to (2p-1) E(L),\\
\var\Big(\frac{S_n\, m^{2(1-p)}}{n}\Big)&=&(2p-1)^2 \var(\frac{S_m}{m^{2p-1}})+{\cal O}\Big(\frac{m}{n}\Big)\sim (2p-1)^2\var(L),
\eeaa
where the random variable $L$ is defined in Theorems 3.7 and 3.8 of \cite{bercu} with
\[E(L)=\frac{2p-1}{\Gamma(2p)} \ttt{and} \var(L)=\frac{1}{(4p-3)\Gamma\big(2(2p-1)\big)}-\frac{(2p-1)^2}{\Gamma^2(2p)}\,.\] 
\end{proposition}
\begin{remark}\label{rem31}\emph{
If $p=1/2$ then $E(S_n)=0$ and $\var(S_n)=n$, since the ERW reduces to a coin-tossing random walk in that case.}\vsb 
\end{remark}

As a main result we now present limit theorems.
\begin{theorem}\label{CLT}
\textbf{\emph{(a)}} If $0<p<3/4$, then
\[ \frac{S_n \sqrt{m}}{n} \dto {\cal N}_{0, (2p-1)^2/(3-4p)}  \ttt{as}\nifi\,.\]
\textbf{\emph{(b)}}  If $p=3/4$, then 
\[  \frac{S_n \sqrt{m/\log m}}{n} \dto {\cal N}_{0, 1/4}  \ttt{as}\nifi\,.\]
\textbf{\emph{(c)}}  If $3/4<p<1$, then
\[ \frac{S_n m^{2(1-p)}}{n} \dto (2p-1)\,L  \ttt{as}\nifi\,,\]
where the random variable $L$ was defined in Theorem 3.7 of \cite{bercu}.
\end{theorem}
\begin{remark}\label{p12}\emph{
If $p=1/2$, then  $S_n/\sqrt{n} \dto {\cal N}_{0,1}$ as $\nifi$ (recall Remark \ref{rem31}). Note also that (a) reduces to $S_n\sqrt{m}/n\pto0$ as $\nifi$.}\vsb
\end{remark}

\subsection{Proofs }
\pf{Proposition \ref{moments}} Let $n>m$. Then
\[E(X_n \,\big|\, {\cal F}_{n-1})=E(X_n\,\big|\, {\cal F}_{m})=(2p-1)\frac{S_m}{m}, \]
and thus
\[E(X_n)=(2p-1)E\Big(\frac{S_m}{m}\Big)\,.\]
This leads to
\begin{equation}\label{condE}E(S_n \,\big|\, {\cal F}_{n-1})=E(S_n \,\big|\, {\cal F}_{m})=S_m +(n-m)(2p-1)\,\frac{S_m}{m}.
\end{equation}
Finally, using that  $E(S_m)=\frac{(2p-1)\Gamma(2p+m-1)}{\Gamma(m)\Gamma(2p)} $ (cf.\ Section 2 in \cite{bercu}), we find that
\begin{eqnarray}\label{m1}
 E(S_n)&=&n(2p-1) \, E(\frac{S_m}{m})+ 2m(1-p)E(\frac{S_m}{m})\nonumber \\
&=& n\frac{(2p-1)\Gamma(m+2p-1)}{\Gamma(m+1)\Gamma(2p)} \Big((2p-1) +2(1-p)\frac{m}{n}\Big)\,,
\end{eqnarray}
where we note that $E(S_m/m)={\cal O}(m^{2p-2})\,.$ \\[1.5mm]

Next, for $m<k<l$,
\[ E(X_k X_l \mid {\cal F}_m)= E(E(X_k X_l\mid {\cal F}_m \cup \sigma(X_k)) \mid{\cal F}_m)
= E(X_k E( X_l\mid{\cal F}_m ) \mid{\cal F}_m)=(2p-1)^2 \frac{S_m^2}{m^2}.\]
For $n>m$, this tells us that
\begin{eqnarray}\label{condVar}
\hskip-2pc E(S_n^2 \mid {\cal F}_m)& =&S_m^2 + 2S_m \sum_{k=m+1}^n E(X_k\mid {\cal F}_m)+\sum_{k=m+1}^n E(X_k^2\mid {\cal F}_m)\nonumber\\
&& \hskip2pc+2\sum_{m+1\le k<\ell\le n}E(X_kX_\ell\, \mid{\cal F}_m)\nonumber\\
&=&\frac{S_m^2}{m^2}\big( m^2+2m(2p-1)(n-m)+(2p-1)^2(n-m)(n-m-1)\big)+(n-m)\nonumber\\
&=& \frac{n^2}{m}\Big(\frac{S_m}{\sqrt{m}}\Big)^2\Big((2p-1)^2-(2p-1)^2\frac{2m+1}{n}+(2p-1)^2\frac{m(m+1)}{n^2}\nonumber\\
&& \hskip2pc +2(2p-1)\Big(\frac{m}{n}-\frac{m^2}{n^2}\Big)+\frac{m^2}{n^2}\Big) + n-m\,.
\end{eqnarray}
For part (a) we thus assume that $0<p<3/4$. Using (\ref{m1}) we obtain, as $\nifi$,
\begin{eqnarray*}
E\Big(\frac{S_n \sqrt{m}}{n}\Big)&=&(2p-1)E\Big(\frac{S_m}{\sqrt{m}}\Big)+2(1-p)\frac{m}{n}E\Big(\frac{S_m}{\sqrt{m}}\Big)={\cal O}\big(m^{2p-3/2}\big)\to 0 \\
E\Big(\Big(\frac{S_n\sqrt{m}}{n}\Big)^2\Big)&=&E\Big(\Big(\frac{S_m}{\sqrt{m}}\Big)^2\Big)\Big((2p-1)^2+(2p-1){\cal O}\big(\frac{m}{n})\Big)\Big) + \frac{m}{n} +o\big(\frac{m}{n}\big).\end{eqnarray*}
Finally, since $\var(S_m/\sqrt{m})\sim 1/(3-4p)$ as $\nifi$ (see \cite{bercu}), it follows that
\[
\var\Big(\frac{S_n\sqrt{m}}{n}\Big)=(2p-1)^2\var\Big(\frac{S_m}{\sqrt{m}}\Big)+ {\cal O}\Big(\frac{m}{n}\Big)\sim\frac{(2p-1)^2}{3-4p}\ttt{as}\nifi,\] 
where an analysis of the calculations above shows that $\var( (\frac{S_n\sqrt{m}}{n}))\sim\frac{m}{n}$  if $p=1/2$.

For part (b)  we use Theorem 3.6 in \cite{bercu} and the calculations above to see that, asymptotically,
\beaa E\Big(\frac{S_n \sqrt{m/\log m}}{n}\Big)&=& \frac12 E(S_m/\sqrt{m\log m}) (1+o(1))= {\cal O}(1/\sqrt{\log m}),\\
 \var\Big(\frac{S_n \sqrt{m/\log m}}{n}\Big)&=&\frac14 \var\Big(\frac{S_m}{ \sqrt{m\log m}}\Big)+o(1)\sim \frac14\,.
\eeaa
The proof of (c) follows along the same lines, together with Theorem 3.7 in \cite{bercu}, where it was shown that  $S_m/m^{2p-1}\stackrel{a.s.}{\to} L$ and $\var(S_m/m^{2p-1}) \to \var(L)$ as $\nifi$, with the random variable $L$ as before. \vsb

Next in turn is the proof of Theorem \ref{CLT}.\\[1.5mm]
\pf{Theorem \ref{CLT}} 
We begin with part (a). Using  (\ref{condE}) we obtain, for $-m\le \ell\le m$,
\[E\Big(\frac{S_n\sqrt{m}}{n}\,\Big| \, S_m=\ell\Big)= \frac{\ell}{\sqrt{m}}\Big((2p-1)+2(1-p)\frac{m}{n}\Big),\]
and from (\ref{condVar})
\[E\Big(\Big(\frac{S_n\sqrt{m}}{n}\Big)^2\,\Big| \, S_m=\ell\Big)=\frac{\ell^2}{m}\Big((2p-1)^2
+{\cal O}\Big(\frac{m}{n}\Big)\Big)+o(1),\]
which in turn, assuming that $\ell =o(\sqrt{n})$, yields
\[\var\Big(\frac{S_n\sqrt{m}}{n}\,\Big| \, S_m=\ell\Big)= {\cal O}\Big( \frac{\ell^2}{n}\Big)= o(1)\,.        \]
Hence, with $\veps_n\searrow 0$ but $\veps^2_n n/m\to \infty$, assuming to begin with that $(2p-1)>0$, and using Chebyshev's inequality, we get
\begin{eqnarray*}
\lefteqn{P\Big( \frac{S_n\sqrt{m}}{n}\le x\Big)}\\&=&\sum_{\ell=-m}^{m} P\Big( \frac{S_n\sqrt{m}}{n}\le x\,\big|\, S_m=\ell\Big)P(S_m=\ell)\\
&=& \sum_{|\ell|<\veps_n\sqrt{n}} P\Big(\frac{S_n\sqrt{m}}{n}-(2p-1)\frac{\ell}{\sqrt{m}}(1+o(1))\le x-(2p-1)\frac{\ell}{\sqrt{m}}(1+o(1))\Big)\times \\&&\hspace*{.5cm}  \times P\Big(\frac{S_m}{\sqrt{m}}=\frac{\ell}{\sqrt{m}}\Big) +\sum_{m\ge|\ell l>\veps_n \sqrt{n}}P\Big( \frac{S_n\sqrt{m}}{n}\le x\,\big|\, S_m=\ell\Big)P(S_m=\ell)\\
&=& \sum_{|\ell|<\veps_n\sqrt{n}}  \big(\ind\{(2p-1)\frac{\ell}{\sqrt{m}}(1+o(1))\le x\}+o(1)\big)\, P\Big(\frac{S_m}{\sqrt{m}}=\frac{\ell}{\sqrt{m}}\Big) +R_n\\
&=& P\Big( -\frac{\veps_n\sqrt{n}}{\sqrt{m}}\le \frac{S_m}{\sqrt{m}}\le \frac{x}{2p-1}\big(1+o(1)\big) \Big)+o(1)+R_n\\
& \to &  {\cal N}_{0,1/(3-4p)}(x/(2p-1))={\cal N}_{0,\frac{(2p-1)^2}{3-4p}}(x)\,,
\end{eqnarray*}
since $\frac{S_m}{\sqrt{m}} \dto {\cal N}_{0,1/(3-4p)}$ and $R_n\le P\big(\big|\frac{S_m}{\sqrt{m}}\big|>\frac{\veps_n\sqrt{n}}{\sqrt{m}}\big)=o(1)$ as $\nifi$.

The case  $2p-1<0$ follows similarly, but the inequality in the indicator function is changing. As the Gaussian distribution is symmetric, we end with the desired result.

Part (b) follows via similar arguments: For $-m\le \ell\le m$,
\[E\Big(\frac{S_n m^{2(1-p)}}{n}\,\Big| \, S_m=\ell\Big)= \frac{\ell}{m^{2p-1}}\Big((2p-1)+2(1-p)\frac{m}{n}\Big)\,.\]
and 
\[E\Big(\Big(\frac{S_n m^{2(1-p)}}{n}\Big)^2\,\Big| \, S_m=\ell\Big)=\frac{\ell^2}{m^{2(2p-1)}}\Big((2p-1)^2+{\cal O}\Big(\frac{m}{n}\Big)\Big)+o(1),\]
which in turn, assuming that $\ell =o\big(\sqrt{n/m^{3-4p}}\big)$, yields
\[\var\Big(\frac{S_n m^{2(1-p)}}{n}\,\Big| \, S_m=\ell\Big)= {\cal O}\Big( \frac{\ell^2 m^{3-4p}}{n}\Big)= o(1)\,.        \]
Hence, with $\veps_n\searrow 0$ but $\veps_n^2 n/m\to \infty$ and $2p-1>0$, the same reasoning tells us that
\begin{eqnarray*}
\lefteqn{P\Big( \frac{S_n m^{2(1-p)}}{n}\le x\Big)}\\&=&\sum_{\ell=-m}^{m} P\Big( \frac{S_n m^{2(1-p)}}{n}\le x\,\big|\, S_m=\ell\Big)P(S_m=\ell)\\
&=& \sum_{|\ell|<\veps_n \sqrt{n/m^{3-4p}} } \Big(\ind\{(2p-1)\frac{\ell}{m^{2p-1}}(1+o(1))\le x\}+o(1)\Big)\, P\Big(\frac{S_m}{m^{2p-1}}=\frac{\ell}{m^{2p-1}}\Big) \\[1mm]
&& +R_n\\[1mm]
&=& P\Big( -\veps_n\,\sqrt{\frac{n}{m}}\le \frac{S_m}{m^{2p-1}}\le \frac{x}{2p-1}\big(1+o(1)\big) \Big)+o(1)+R_n\\
& \to &  F_L(x/(2p-1)\pm)\stackrel{d}{=}(2p-1)\,F_L(x\pm)\,,
\end{eqnarray*}
where we use Theorem 3.7 in \cite{bercu} and the fact that $R_n \le P\Big( \Big|\frac{S_m}{m^{2p-1}}\Big| > \veps_n \sqrt{\frac{n}{m}}\Big) \to 0 \ttt{as} \nifi$. Note that we do not know whether $L$ is a continuous random variable or not, but this does not affect the conclusion.

Part (c) runs along the same lines, noticing that $\frac{S_m}{\sqrt{m\log m }}\asto {\cal N}_{0,1}$ as $\nifi$.\vsb

\section{ERW:s with delays}
\setcounter{equation}{0}
In this section we introduce the possibility of delays, in that the elephant, in addition, has a choice of staying put in every step. Our aim is  to extend the results above to this three-point case. More precisely, the steps are defined via
\bea\label{23}\xnp=\begin{cases} +X_{K},\ttt{with probability} p\in[0,1], \\-X_{K},\ttt{with probability} q\in[0,1],\\\phantom{+)}0,\ttt{\,\,  with probability}r\in[0,1],\end{cases}\eea
where  $p+q+r=1$, and where $K$ has a uniform distribution on the integers $1,2,\ldots,n$.  Everything reduces, of course, to our previous results if if $r=0$.
We refer to the paper \cite{113} for results in the classical case, as well as to cases with finite memories.\\

\begin{proposition}\label{delmom}
Let $m,n\to\infty$ in such a way that $m/n\to0$.\\[1mm]
\textbf{\emph{(a)}} If $ p-q<1/2$, then
\beaa
E(\frac{S_n\sqrt{m}}{n})&=&(p-q)E(\frac{S_m}{\sqrt{m}})+(1+q-p)\frac{m}{n} E(\frac{S_m} {\sqrt{m}})\to 0,\\
\var(\frac{S_n\sqrt{m}}{n})&=&(p-q)^2 \var(\frac{S_m}{\sqrt{m}})
+{\cal O}\Big(\frac{m}{n}\Big)=\frac{(p^2-q^2)^2}{1-2(p-q)} +o(1).
\eeaa
\textbf{\emph{(b)}} If $p-q=1/2$, then 
\beaa
E(\frac{S_n\sqrt{m/\log m}}{n})&=&(p-q)E(\frac{S_m}{\sqrt{m \log m}})+(1+q-p)\frac{m}{n} E(\frac{S_m} {\sqrt{m\log m}})\to 0,\\
\var(\frac{S_n\sqrt{m/ \log m}}{n})&=&(p-q)^2 \var(\frac{S_m}{\sqrt{m \log m}})
+{\cal O}\Big(\frac{m}{n\log m}\Big)=\frac14 (p+q)^2 +o(1).
\eeaa
\textbf{\emph{(c)}} If $p-q>1/2$, then
\beaa
E(\frac{S_n\,m^{(1+q-p)}}{n})&=&(p-q)E(\frac{S_m}{m^{p-q}})+(1+q-p)\frac{m}{n} E(\frac{S_m }{m^{p-q}})\to (p-q) E(L),\\
\var(\frac{S_n\, m^{(1+q-p)}}{n})&=&(p-q)^2 \var(\frac{S_m}{m^{p-q}})+{\cal O}\Big(\frac{m}{n}\Big)\sim (p-q)^2\var(L),
\eeaa
where the random variable  $L$ is defined in Theorem 4.1 of \cite{113} with moments
\[E(L)=\frac{(p-q)}{\Gamma(1+p-q)} \ttt{and} \var(L)=\frac{(p+q)^2}{(2(p-q)-1)\Gamma((2(p-q))}-\frac{(p-q)^2}{\Gamma^2(1+p-q)}\,.\] 
\end{proposition}
\begin{remark}\emph{If $p=q$ then $\var(S_n/\sqrt{n})\to 1$ as $\nifi$.}\vsb
\end{remark}
Furthermore we have the following asymptotic distributions.
\begin{theorem}\label{CLT2}\textbf{\emph{(a)}} If $p-q<1/2$, then
\[ \frac{S_n \sqrt{m}}{n}\dto (p+q){\cal N}_{0,\frac{(p-q)^2(p+q)}{1-2(p-q)}}+r \delta_0 \ttt{as} \nifi\,.\]
\textbf{\emph{(b)}} If $p-q=1/2$, then 
\[ \frac{S_n \sqrt{m/\log m}}{n}\dto (p+q){\cal N}_{0,\frac14 (p+q)}+r \delta_0 \ttt{as} \nifi\,.\]
\textbf{\emph{(c)}} If $p-q>1/2$, then
\[ \frac{S_n m^{1+q-p}}{n}\dto (p-q) L  \ttt{as} \nifi,\]
with the random variable $L$ as above.
\end{theorem}
\begin{remark}\label{pgleichq}
\emph{If $p=q$ then we have $\frac{S_n}{\sqrt{n}} \dto 2p\,{\cal N}_{0,1}+(1-2p)\delta_0$ as $\nifi$.}\vsb
\end{remark}
\subsection{Proofs}

\pf{Proposition \ref{delmom}} For $n>m$,
\[ E(X_n\,\big|\, \cfm)=\frac{(p-q)S_m}{m}\,,\]
and so
\[ E(S_n\,\big|\, \cfm)=S_m+(n-m)\frac{(p-q)S_m}{m}=\frac{S_m}{m}\big(n(p-q)+m(1+q-p)\big)\,.\]
As before we learn that, for $k,\ell>m$,
\[ E(X_kX_\ell \,\big|\, \cfm)=\frac{(p-q)^2S^2_m}{m^2}\,.\]
Thus,
\begin{eqnarray*}
E(S_n^2\,\big|\, \cfm)&=&S_m^2 + 2S_m \sum_{k=m+1}^n E(X_k\, \big|\, {\cal F}_m)+\sum_{k=m+1}^n E(X_k^2\, \big|\, {\cal F}_m)\\
&& \hskip2pc+2\sum_{m+1\le k<\ell\le n}E(X_kX_\ell\, \big|\, {\cal F}_m)\\
&=&\frac{S_m^2}{m^2}\Big( m^2+2m(p-q)(n-m)+(p-q)^2(n-m)(n-m-1)\Big)+(n-m)(p+q)\\
&=& \frac{n^2}{m}\Big(\frac{S_m}{\sqrt{m}}\Big)^2\Big((p-q)^2-(p-q)^2\frac{2m+1}{n}+(p-q)^2\frac{m(m+1)}{n^2}\\
&&\hskip2pc  +2(p-q)\Big(\frac{m}{n}-\frac{m^2}{n^2}\Big)+\frac{m^2}{n^2}\Big) +\frac{n^2}{m}(p+q)\Big(\frac{m}{n} -\frac{m^2}{n^2}\Big).
\end{eqnarray*}
The remaining calculations follow as in the proofs above for the ordinary ERW, together with the results in \cite{113}.\vsb

Next we proceed with the proof of Thm.\ref{CLT2}.

\pf{Theorem \ref{CLT2}} 
We deal with part (a) and obtain, for $-m\le \ell\le m$,
\[E\Big(\frac{S_n\sqrt{m}}{n}\,\big| \, S_m=\ell\Big)= \frac{\ell}{\sqrt{m}}\Big((p-q)+(1+q-p)\frac{\sqrt{m}}{n}\Big)\,.\]
and 
\[E\Big(\Big(\frac{S_n\sqrt{m}}{n}\Big)^2\,\big| \, S_m=\ell\Big)=\frac{\ell^2}{m}\Big((p-q)^2+{\cal O}\Big(\frac{m}{n}\Big)\Big)+o(1),\]
which, assuming that $\ell =o(\sqrt{n})$, in turn, yields
\[\var\Big(\Big(\frac{S_n\sqrt{m}}{n}\Big)^2\,\big| \, S_m=\ell\Big)= {\cal O}\Big( \frac{\ell^2}{n}\Big)= o(1)\,.        \]
Next, let $\veps_n\searrow 0$, such that $m/(\veps^2_n n)\to 0$, and assume for the moment that $p-q>0$. Using Chebyshev's inequality yields
\begin{eqnarray*}
P\Big( \frac{S_n\sqrt{m}}{n}\le x\Big)&=&\sum_{\ell=-m}^{m} P\Big( \frac{S_n\sqrt{m}}{n}\le x\,\big|\, S_m=\ell\Big)P(S_m=\ell)\\
&=& \sum_{|\ell|<\veps_n\sqrt{n}} P\Big(\frac{S_n\sqrt{m}}{n}-(p-q)\frac{\ell}{\sqrt{m}}(1+o(1))\le x-(p-q)\frac{\ell}{\sqrt{m}}(1+o(1))\Big) \\
&&\hskip2pc  \times P\Big(\frac{S_m}{\sqrt{m}}=\frac{\ell}{\sqrt{m}}\Big) 
+\sum_{\veps_n\sqrt{n}<|\ell|\leq m }P\Big( \frac{S_n\sqrt{m}}{n}\le x\,\big|\, S_m=\ell\Big)P(S_m=\ell)\\
&=& \sum_{|\ell|<\veps_n\sqrt{n}}  \big(\ind\{(p-q)\frac{\ell}{\sqrt{m}}(1+o(1))\le x\}+o(1)\big)\, P\Big(\frac{S_m}{\sqrt{m}}=\frac{\ell}{\sqrt{m}}\Big) +R_n\\
&=& P\Big( -\frac{\veps_n\sqrt{n}}{\sqrt{m}}\le \frac{S_m}{\sqrt{m}}\le \frac{x}{p-q}\big(1+o(1)\big) \Big)+o(1)+R_n\\[1.5mm]
& \to &  (p+q){\cal N}_{0,(p+q)/(1-2(p-q))}(x/(p-q))+\delta_0(x/(p-q))\\[1.5mm]
&=&(p+q){\cal N}_{0,\frac{(p-q)^2(p+q)}{1-2(p-q)}}(x)+r\delta_0(x)\,,
\end{eqnarray*}
since, as $\nifi$,
\[\frac{S_m}{\sqrt{m}} \dto(p+q) {\cal N}_{0,(p+q)/(1-2(p-q))}+\delta_0(x)\ttt{and}R_n\le P\big(\big|\frac{S_m}{\sqrt{m}}\big|>\frac{\veps_n\sqrt{n}}{\sqrt{m}}\big)=o(1).\] 
The other cases follow along the same arguments.\vsb

The situation in Remark \ref{pgleichq} is comparable to that of Remark \ref{p12}, and the result follows from the classical CLT if $X_1=\pm 1$, and the fact that  $S_n=0\,,\; \forall n$ if $X_1=0$.

\subsection{The number of zeros in ERW:s with delays}
In an ERW with delays the number of zeros is dominating. Technically it is easier to investigate the number of non-zeros.
Therefore, let
\[ I^*_n=\ind_{\{X_n\not= 0\}}  \ttt{and}  N^*_n=\sum_{k=1}^n I^*_k\,.\]
The total number of zeros obiviously equals $N_n=n-N_n^*$. For $N_n^*$ we have the following moments.
\begin{proposition}\label{non0}
For $0<r<1$ we have, as $\nifi$,
\beaa
 E\Big(\frac{N_n^* m^r}{n}\Big)&=& \frac{m^r \Gamma(m+1-r)}{\Gamma(1-r)\Gamma(m+1)}\big( (1-r)+mr/n)
\sim \frac{1-r}{\Gamma(1-r)}+{\cal O}(m/n)+{\cal O}(1/m);\\
 E\Big(\Big(\frac{N_n^* m^r}{n}\Big)^2\Big)&=& (1-r)^2\,d_r+{\cal O}(m/n),
\eeaa
where the constant $d_r$ was defined in \cite{114}.
\end{proposition}
\begin{proof}  For $n>m$ we find that
\[ E(N_n^*\mid \cfn)=E(N_n^*\mid \cfm)=N_m^*+\sum_{k=m+1}^n E(I_k^*\mid \cfm)=\frac{N_m^*}{m}\big(n(1-r)+mr\big).\]
Exploiting formula (3.1) of \cite{114} yields
\[E(N_n^*)=\frac{\Gamma(m+1-r)}{\Gamma(1-r)\Gamma(m+1)}\big(n(1-r)+mr\big)\sim m^{-r}n \Big(\frac{1-r}{\Gamma(1-r)}+r \frac{m}{n}+{\cal O}\big(m^{-1}\big)\Big)\,,\]
which establishes the first relation.

As for the second moment,
\begin{eqnarray*}
\lefteqn{E((N_n^*)^2\mid \cfn)= E((N_n^*)^2\mid \cfm)=}\\
&=&(N_m^*)^2+2N_m^*\sum_{k=m+1}^n E(I_k^*\mid \cfm) +\sum_{k=m+1}^n E((I_k^*)^2\mid \cfm)+ 2 \sum_{m<k<\ell\le n} E(I_k^*I_\ell^*\mid\cfm\, )\\
&=&(N_m^*)^2+ 2(1-r)(n-m) \frac{(N_m^*)^2}{m}+(1-r)(n-m)  \frac{N_m^*}{m}\\[1.4mm]
&&\hspace*{4cm}+2(1-r)^2 \frac{(N_m^*)^2}{m^2}\frac{ (n-m)(n-m-1)}{2}\,.
\end{eqnarray*}
Consequently,
\begin{eqnarray*}
E((N_n^*)^2)&=& E\Big(\frac{(N_m^*)^2}{m^2}\Big)\Big(m^2+2(1-r)(nm-m^2)+(1-r)^2\big(n^2-n(2m+1)+m(m+1)\big)\Big) \\
&&\hskip2pc +E\Big(\frac{N_m^*}{m}\Big)(1-r)(n-m)\\
&=&E\Big(\frac{(N_m^*)^2}{m^2}\Big)\Big(\big(n+(1-r)(n-m)\big)^2-(1-r)^2(n-m)\Big)+E\Big(\frac{N_m^*}{m}\Big)(1-r)(n-m),
\end{eqnarray*}
which leads to the desired result by applying the fact that $E\big((N_n^*)^2\big)\sim d_r m^{2(1-r)}$ from \cite{114}.\vsb
\end{proof}
\begin{theorem}\label{zeros}
For $0<r<1$,
\[ \frac{S_n m^r}{n}\dto \frac{Y}{ \Gamma(1-r)}\ttt{as}\nifi,\]
where the random variable $Y$ is defined in Theorem 3.1 of \cite{114}.
\end{theorem}
\begin{proof}
As before we begin with conditional moments:
\[ E\Big( \frac{N_n m^r}{n}\Bigm| N^*_m=\ell\Big)=\frac{\ell}{m^{1-r}}\big((1-r)+rm/n\big) \ttt{and} \]
\[\var\Big( \frac{N^*_n m^r}{n}\Bigm| N^*_m=\ell \Big)={\cal O}\Big(\Big(\frac{\ell}{m^{1-r}}\Big)^2\frac{m}{n}\Big)=o(1)\,,\]
as long as $\ell =o\big(\sqrt{n/m^{2r-1}}\big)\,.$
We continue using a null-sequence $\veps_n$, such that $\veps_n^2\,n/m \to \infty $ as $\nifi$. Then
\begin{eqnarray*}
\lefteqn{P\Big( \frac{N^*_n m^r}{n}\le x\Big)=}\\&=&\sum_{k=0}^n  P\Big(\frac{N^*_n m^r}{n}\le x\Bigm|N^*_m=k\Big)P(N_m^*=k)\\
&=&\sum_{0\le k \le \veps_n \sqrt{n/m^{2r-1}} } P\Big(\frac{N^*_n m^r}{n}-\frac{k}{m^{1-r}}\,(1-r)\le x -\frac{k}{m^{1-r}}\,(1-r) \,\big|\, N^*_m=k \Big) P(N_m^*=k)+R_n\\
&=&\sum_{k/m^{1-r}\le x/(1-r)}P\Big(\frac{N_m^*}{m^{1-r}}=\frac{k}{m^{1-r}}\Big)+o(1)+R_n\\
&\to &F_Y\big(x \Gamma(2-r)/(1-r)\big)=F_Y\big(x\,\Gamma(1-r)\big)\,,
\end{eqnarray*}
in view of Theorem 3.1 in \cite{111}, and the fact that $R_n \le P \big(N_n^*/m^{1-r}> \sqrt{\veps_n^2 n/m}\big) \to 0$ as $\nifi$ by the same result.\vsb
\end{proof}

\section{Including also the last step;\quad $\memn = \{1,2,\ldots,m_n,n\}$}
\setcounter{equation}{0}
If the elephant, in addition, remembers the most recent step it turns out that the results of the previous section remain true without change. This is reasonable, since the probability that the most recent step is, indeed, the one that is chosen is $1/m\to 0$ as $\nifi$. However, since the sum of the probabilities is divergent, the second Borel--Cantelli lemma is not applicable, and some other argument must be exploited.

Toward that end, we create an ERW that selects precisely the steps that are based on the most recent step. We thus define 
\[I_k = \ind\{X_k = \pm \xkm\}\ttt{and} M_n = \min\{\ell: \sum_{k=1}^\ell I_k=n\},\]
and set
\[ Y_k=X_k\cdot\{I_k=1\} \ttt{and}T_n=\sum_{k=1}^{M_n} Y_k \ge \sum_{k=1}^n Y_k,\quad n\geq1.\]
The sequence $\{T_n,\,n\geq1\}$ thus defined has the desired property, which takes us back to Section 7 of \cite{111}, according to which $\frac{T_n}{\sqrt{n}}\dto  {\cal N}_{0, \frac{p}{1-p}}$ as $\nifi$. Since the additional factors needed for the expressions in Theorem \ref{CLT} all tend to zero as $n$ increases (and $m/n\to0$), we have established the following fact.

\begin{theorem} The conclusions of Theorem \ref{CLT} remain true for the case $\memn = \{1,2,\ldots,m_n,n\}$.
\end{theorem}

\begin{remark}\emph{Another glance at Section 7 of \cite{111} shows that the results in Proposition \ref{moments} also remain unchanged in the current situation.}\end{remark}
\begin{remark}\emph{If the elephant remembers the first $m_n$ steps and two most recent ones, the results of \cite{111}, Section 8, similarly tell us that the results from Section \ref{drei} remain unchanged. In the cited paper we also suggested, in Remark 8.2, that a similar result shoud hold if the elephant remembers the  $k$ most recent steps for some fixed number $k \in \mathbb{N}$. This, in turn, suggests that the results of Section \ref{drei} remain unchanged for
$\memn = \{1,2,\ldots,m_n, n-k, n-k+1,\ldots,n\}$ for any fixed $k$, and, in fact, also if
$\memn = \{1,2,\ldots,m_n, n-k_n, n-k_n+1,\ldots,n\}$ whenever $k_n/m_n\to0$ as $\nifi$.\vsb}
\end{remark}

An inspection of the delayed case shows that the same conclusions hold in that case; in fact, even more so, since the corresponding $T_n$-process terminates after a finite (geometric) number of steps (\cite{113}, Theorem 5.1). For completeness (and possible reference) we state the corresponding result separately, but omit the analogous remark.

\begin{theorem} The conclusions of Theorem \ref{CLT2} remain true for the case $\memn = \{1,2,\ldots,m_n,n\}$.
\end{theorem}

\section{The case $\lim_{\nifi} m/n=\alpha\in (0,1]$}
\setcounter{equation}{0}
\begin{proposition}\label{momentsc}
Let $m,n\to\infty$, such that $m/n \to \alpha \in (0,1]$. \\[2mm]
\textbf{\emph{(a)}} If $0<p<3/4$, then
\beaa 
E\Big(\frac{S_n\sqrt{m}}{n}\Big)&=&E\Big(\frac{S_m}{\sqrt{m}}\Big)\cdot\big((2p-1)+2\alpha\,(1-p)\big)+o(1) \to 0,\\
\var\Big(\frac{S_n\sqrt{m}}{n}\Big)&=&\frac{1}{3-4p}\cdot\big((2p-1)^2(1-\alpha)^2 +2(2p-1)\alpha(1-\alpha)+\alpha^2\big) +\alpha(1-\alpha)+o(1)\\
&=& \frac{1}{3-4p}\cdot\big((2p-1)(1-\alpha)+\alpha)\big)^2 +\alpha(1-\alpha)+o(1).
\eeaa
If, in particular, $p=1/2$, then $\var\big(\frac{S_n\sqrt{m}}{n}\big) \to \alpha$, or, equivalently, $\var\Big(\frac{S_n}{\sqrt{n}}\Big) \to 1$.\\[2mm]
\textbf{\emph{(b)}} If $p=3/4$, then
\beaa
E\Big(\frac{S_n\sqrt{m/\log m}}{n}\Big)&=&E\Big(\frac{S_m}{\sqrt{m\log m}}\Big) \cdot\big(1+\alpha\big)/2+o(1)\to0,\\
\var\Big(\frac{S_n\sqrt{m/\log m}}{n}\Big)&=&\var\Big(\frac{S_m}{\sqrt{m\log m}}\Big)\cdot \Big( \frac14(1-\alpha)^2 
+\alpha\Big) +o(1)=\frac14(1-\alpha)^2 +\alpha +o(1).
\eeaa
\textbf{\emph{(c)}} If $3/4<p<1$, then
\beaa E\Big(\frac{S_n m^{2(1-p)}}{n}\Big)&=&E(\frac{S_m}{m^{2(1-p)}})\cdot\big((2p-1)
+2\alpha(1-p)+o(1) \big)\to 0,\\
\var\Big(\frac{S_n m^{2(1-p)}}{n}\Big)&=&\var\big(L\big)\cdot\big((2p-1)^2(1-\alpha)^2 +2(2p-1)\alpha(1-\alpha)+\alpha^2\big) +o(1)\\
&=&\var\big(L\big)\cdot\big((2p-1)(1-\alpha)+\alpha\big)^2 +o(1).
\eeaa
\vsp
\end{proposition}
We note, in particular, that if $\alpha=1$, then the results coincide asymptotically with those in our predecessors. This is, of course, no surprise, since it becomes more and more unlikely that the elephant does \emph{not\/} base a new step on the previous one.

We guess that there are limit distributions along the lines of Proposition \ref{CLT}  but our technique of proof fails as the conditional variances do not vanish asymptotically.

\section{Further remarks}
\setcounter{equation}{0}
\begin{itemize}
\item In \cite{111} we noted that for fixed, finite, memories $\memn=\{1,2,\dots,m\}$, there was no phase transition, in contrast to the case of full memory (recall \cite{bercu}). However there was no simple asymptotic normality. Our present results show, with appropriate variances, that once the memory of the past is gradually increasing (to infinity) asymptotic normality and phase transition occur.
\item In \cite{111} we have shown that $S_n/\sqrt{n}$ is asymptotically normal when $\memn=\{n-m, n-m+1,\dots,n\}$ with fixed $m$. The asymptotic variance for arbirtrary $m$ was calculated in \cite{Brasil} to be $\sigma^2=\frac{m-1+2p}{2(1-p)(2(1-p)m+2p-1)}$. We conjecture that asymptotic normality still holds if $\memn=\{n-m_n, n-m_n+1,\dots, n\}$ with increasing $m_n$ or also if
$\memn = \{1,2,\ldots,m_n, n-k_n, n-k_n+1,\ldots,n\}$ whenever $k_n/m_n\to0$ as $\nifi$. We also conjecture that  if $ m_n$ does not increase too rapidly there are no phase transitions and that the asymptotic variance is given by $\sigma^2=1/(4(1-p)^2)$.
\end{itemize}

\medskip\noindent {\small Allan Gut, Department of Mathematics,
Uppsala University, Box 480, SE-751\,06 Uppsala, Sweden;\\
Email:\quad \texttt{allan.gut@math.uu.se}\\
URL:\quad \texttt{http://www.math.uu.se/\~{}allan}}
\\[4pt]
{\small Ulrich Stadtm\"uller, Ulm University, Department of Number
Theory
and Probability Theory,\\ D-89069 Ulm, Germany;\\
Email:\quad \texttt{ulrich.stadtmueller@uni-ulm.de}\\
URL:\quad
\texttt{https://www.uni-ulm.de/en/stadtmueller/}}

\end{document}